\documentclass[10pt,a4paper,twoside]{article}
\usepackage[utf8]{inputenc} 
\usepackage[english]{babel} 
\usepackage{amsfonts,amsmath,amssymb,amscd}
\usepackage{bezier,graphics,graphicx}
\usepackage{epsfig, layout,latexsym}
\usepackage{theorem}
\usepackage{cite}
\usepackage{indentfirst}
\usepackage{array}

\usepackage[margin=1in]{geometry}

\theorembodyfont{\sl}
\newtheorem{theorem}{Theorem}[section]

\newtheorem{lemma}[theorem]{Lemma}
\newtheorem{corollary}[theorem]{Corollary}
\newtheorem{remark}[theorem]{Remark}

\numberwithin{equation}{section}
\numberwithin{theorem}{section}

\newcommand{\esup}{\text{ess\,sup}}

\begin{document}

\begin{center}
{\bf Numerical Differentiation of Functions of Two Variables \\ 
Using Chebyshev Polynomials}
\end{center}

\vspace*{5mm}
\centerline{\textsc{M. Kyselov$\!\!{}^{\dag}$}, \textsc{S.G. Solodky  $\!\!{}^{\dag,\ddag}$} }

\vspace*{5mm}
\centerline{$\!\!{}^{\dag}\!\!$ Institute of Mathematics, National Academy of Sciences of Ukraine, Kyiv}
\centerline{$\!\!{}^{\ddag}\!\!$ University of Giessen, Department of Mathematics, Giessen, Germany}

\begin{abstract}
We investigate the problem of numerical differentiation of bivariate functions from weighted Wiener classes
using Chebyshev polynomial expansions.
We develop and analyze a new version of the truncation method based on Chebyshev polynomials and
the idea of hyperbolic cross to reconstruct partial derivatives of arbitrary order.
The method exploits the approximation properties of Chebyshev polynomials and their natural connection
to weighted spaces through the Chebyshev weight function.
We derive a choice rule for the truncation parameter as a function of the noise level,
smoothness parameters of the function class, and the order of differentiation.
This approach allows us to establish explicit error estimates in both weighted integral norms and
uniform metric.
\vspace*{3mm}

\vspace*{3mm}

\end{abstract}

\section{Introduction. Description of the problem}
The problem of numerical differentiation represents one of the fundamental challenges in computational mathematics, with extensive applications spanning scientific computing, engineering analysis, and data processing. This challenge becomes particularly significant when dealing with situations where analytical differentiation of functions is computationally infeasible or when working with data corrupted by measurement errors. Despite decades of theoretical development and practical implementation, numerical differentiation continues to attract substantial research attention (see, for example, \cite{Cul71}, \cite{And84}, \cite{Groetsch_1992_V74_N2}, \cite{Qu96}, \cite{Hanke&Scherzer_2001_V108_N6}, \cite{RammSmir_2001}, \cite{Ahn&Choi&Ramm_2006}, \cite{Wang_Hon_Ch_2006}, \cite{Nakamura&Wang&Wang_2008}, \cite{Dutt_1996}, \cite{SSS_CMAM}, \cite{Singh_2023}, \cite{Kiltz2022}, \cite{Abdelhakem_2020}).

The inherent difficulty of numerical differentiation stems from its ill-posed nature in the sense of Hadamard. Specifically, small perturbations in the input data can lead to arbitrarily large errors in the computed derivatives, making the problem fundamentally unstable. This instability necessitates the development of regularization principles when developing differentiation methods to ensure stable approximations while maintaining an acceptable level of accuracy. Therefore, one such method, called the truncation method,  attracts considerable attention from specialists.

It should be noted that some versions of the truncation method for numerical differentiation
(primarily Legendre truncation methods) have been studied in various contexts (see \cite{Abdelhakem_2020}, \cite{Lu_Naum_Per}).
The approach demonstrates several favorable properties, including straightforward algorithmic implementation and
the absence of requirements for solving complex equation systems.
This prompted further research into the approximation properties of the method across various function classes
(see, for example, \cite{Sol_Stas_UMZ2022}, \cite{Sem_Sol_ETNA}, \cite{Sem_Sol_SIM2025}, \cite{Sol_Stas_JoC2024}).

A distinctive feature of this work is the study of the truncated Chebyshev method for numerical differentiation problems of bivariate functions. Although truncation methods can be applied to various polynomial bases, the use of Chebyshev polynomials offers several distinct advantages. The Chebyshev basis is particularly well-suited for approximation on the interval $[-1,1]$ due to its minimax properties and excellent convergence compared to other polynomial bases.

The primary goal of this study is to establish error bounds for the truncated Chebyshev method when applied to bivariate functions from Wiener weight classes to recover their partial derivatives of arbitrary order. Unlike previous studies, which focused primarily on first-order derivatives and specific metric spaces, here we conduct a comprehensive analysis over a wide range of parameters characterizing the function class, the order of the derivative, and the metrics of both the input and output spaces. The use of Chebyshev polynomials with their corresponding weight function introduces additional technical considerations that must be carefully considered in the error analysis.

In developing our approach, we pay special attention to determining the optimal value of the truncation parameter that minimizes the total error arising from both series truncation and data perturbations. We derive explicit relationships between this parameter and the noise level in the input data, the smoothness characteristics of the function, and the order of the recovered derivative. These relationships provide practical recommendations for implementing the method in computational applications.

Our results contribute to both the theoretical understanding of numerical differentiation in weighted spaces and the development of practical algorithms for derivative recovery. The analysis extends existing results to cover partial derivatives of arbitrary order and various output metrics, providing a comprehensive characterization of the method's capabilities and limitations. The specific properties of the Chebyshev basis allow us to obtain sharp (in order) error estimates that reflect the compromise between the approximation properties of the truncation method and the perturbation level in the input data.

To establish the mathematical framework for our analysis, we now introduce the necessary notation and fundamental concepts.

By $L_{q,\omega}=L_{q,\omega}(Q)$, $0 < q < \infty$, we mean weighted Lebesgue spaces of real-valued functions $f(t,\tau)$ defined on $Q=[-1,1]^2$ with the Chebyshev weight function $\omega(t,\tau)=(1-t^2)^{-1/2}(1-\tau^2)^{-1/2}$ (see \cite{MasonHandscomb2002} for properties of Chebyshev polynomials), equipped with the norms
$$
\|f\|_{L_{q,\omega}} = \left(\int_{Q} \omega(t,\tau)\, |f(t,\tau)|^q\, d\tau dt\right)^{1/q} < \infty.
$$

For $q = \infty$, we define $L_{\infty,\omega}=L_{\infty,\omega}(Q)$ as the space of measurable and bounded almost everywhere functions  with the norm
$$
\|f\|_{L_{\infty,\omega}} = \esup_{\substack{(t,\tau) \in Q}} |f(t,\tau)| < \infty.
$$

Let $\{\bar{T}_k(t)\}_{k=0}^\infty$ be the system of Chebyshev polynomials of the first kind as
$$
\bar{T}_k(t) = \cos(k\arccos t), \quad k=0,1,2,\ldots,
$$
defined on $[-1,1]$.

We proceed from $\{\bar{T}_k(t)\}$ to an orthonormal system $\{T_k(t)\}$ of Chebyshev polynomials according to the relations:
$$
T_0(t)=\frac{1}{\sqrt{\pi}} \bar{T}_0(t),\qquad
T_k(t)=\frac{\sqrt{2}}{\sqrt{\pi}} \bar{T}_k(t),\quad k\in \mathbb{N}.
$$

Note that for $q=2$, the space $L_{2,\omega}$ becomes a weighted Hilbert space with the inner product

$$
\langle f, g\rangle=\int_{Q} \omega(t,\tau)\, f(t,\tau)\, g(t,\tau)\, d\tau dt
$$
and standard norm
$$
\|f\|_{L_{2,\omega}}^2=\sum_{k,j=0}^{\infty}|\langle f, T_{k,j} \rangle|^2 < \infty ,
$$
where
$$
\langle f, T_{k,j}\rangle=\int_{Q} \omega(t,\tau)\, f(t,\tau)\, T_k(t)\, T_j(\tau)\, d\tau dt, \quad k,j=0,1,2,\ldots,
$$
are Fourier-Chebyshev coefficients of $f$, and $T_{k,j}(t,\tau) = T_k(t)T_j(\tau)$ are tensor products of one-dimensional Chebyshev polynomials.

Moreover, let $C=C(Q)$ be the space of continuous on $Q$ real-valued bivariate functions
with bounded norms and $\ell_p$, $1\leq p\leq\infty$, be the space of numerical sequences
$\overline{x}=\{x_{k,j}\}_{k,j\in\mathbb{N}_0}$, $\mathbb{N}_0=\{0\}\bigcup\mathbb{N}$,
such that the corresponding relation
$$
\|\overline{x}\|_{\ell_p}  := \left\{
\begin{array}{cl}
\left(\sum\limits_{k,j\in\mathbb{N}_0} |x_{k,j}|^p\right)^{\frac{1}{p}} < \infty ,
 & 1\leq p<\infty ,
\\
\sup\limits_{k,j\in\mathbb{N}_0}  |x_{k,j}| < \infty ,
  & p=\infty ,
\end{array}
\right.
$$
is fulfilled.

We introduce the space of functions
$$
 W_{s,2}^{\overline{\mu}} := W_{s,2}^{\overline{\mu}}(Q) := \left\{f\in L_{2,\omega}(Q): \|f\|_{s,\overline{\mu}}^{s} =
 \sum\limits_{k,j=0}^{\infty} {\underline{k}}^{s\mu_1} {\underline{j}}^{s\mu_2}
 |\langle f , T_{k,j} \rangle |^s < \infty
   \right\} ,
$$
where $\overline{\mu}=(\mu_1,\mu_2)$, $\mu_1, \mu_2 > 0$, $1\le s<\infty$, $\underline{k}=\max\{1,k\}$,
$k=0,1,2,\dots$.
These functional spaces $W_{s,2}^{\overline{\mu}}$ are extensively studied in approximation theory and are recognized as weighted Wiener classes, as discussed in \cite{Kolom2023}.

By $BW_{s,2}^{\overline{\mu}} = BW_{s,2}^{\overline{\mu}}(Q) = \{f\in W_{s,2}^{\overline{\mu}}: \|f\|_{s,\overline{\mu}} \leq 1\}$
we denote a unit ball (class) from $W_{s,2}^{\overline{\mu}}$.

It should be noted that
$BW^{\overline{\mu}}_{s,2}$ is a generalization of the class of bivariate functions with dominating mixed derivatives.

We represent a function $f(t,\tau)$ from $W_{s,2}^{\overline{\mu}}$ as
$$
 f(t,\tau) = \sum_{k,j=0}^{\infty} \langle f, T_{k,j}\rangle T_k(t)T_j(\tau),
$$
and by $r$-th partial derivative of $f$ we mean the following series
\begin{equation}\label{r_deriv}
f^{(r,0)}(t,\tau) =  \sum_{k=r}^{\infty} \sum_{j=0}^{\infty} \langle f, T_{k,j}\rangle
T^{(r)}_k(t)T_j(\tau), \quad r=1,2,\ldots.
\end{equation}

Moreover, it holds true
\begin{equation}\label{max_cheb}
\max_{-1\le t\le 1} |T_k(t)| = T_k(1) := \left\{
\begin{array}{cl}
1/\sqrt{\pi}, & k=0 ,\\
\sqrt{2/\pi}, & k\in \mathbb{N} .
\end{array}
\right.
\end{equation}

Suppose that instead of the function $f\in W_{s,2}^{\overline{\mu}}$
we only know some its perturbation $f^\delta \in L_{2,\omega}$
with error $\delta$ in the metric of $\ell_p$, $1\leq p\leq\infty$.
More accurately, we assume that there is a sequence of numbers $\overline{f^\delta}= \{\langle
f^\delta, T_{k,j} \rangle\}_{k,j\in\mathbb{N}_0}$ such that for $\overline{\xi}=
\{\xi_{k,j}\}_{k,j\in\mathbb{N}_0}$, where $\xi_{k,j}=\langle f-f_\delta,T_{k,j}\rangle$, and for some $1\leq
p\leq \infty$ the relation
\begin{equation}\label{perturbation}
   \|\overline{\xi}\|_{\ell_p} \leq \delta , \quad 0<\delta <1 ,
\end{equation}
is true.

\section{Truncation method. Error estimation in $L_{2,\omega}$ metric}

As an algorithm for the numerical differentiation of functions from $W_{s,2}^{\overline{\mu}}(Q)$ we will use
a truncation method. The essence of this method is to replace the Fourier series (\ref{r_deriv}) with a finite
Fourier sum using perturbed data $\langle f^\delta, T_{k,j} \rangle$. In the truncation method to ensure the
stability of the approximation and achieve the required order of accuracy, it is necessary to choose properly the
discretization parameter, which here serves as a regularization parameter. So, the process of regularization in the method
under consideration consists in matching the discretization parameter with the perturbation level $\delta$ of the input
data. The simplicity of implementation is the main advantage of this method.

In the case of an arbitrary finite domain $\Omega$ of the coordinate plane $[r,\infty)\times[0,\infty)$, the
truncation method for differentiation of bivariate functions has the form
$$
\mathcal{D}_\Omega^{(r,0)} f^\delta(t,\tau)=\sum_{(k,j)\in \Omega} \langle f^\delta, T_{k,j} \rangle
T_k^{(r)}(t)T_j(\tau).
$$
To increase the efficiency of the approach under study, as the domain $\Omega$ we take a hyperbolic cross of
the following form
$$
\Omega = \Gamma_{n,\gamma} := \{ (k,j):\ k\cdot j^{\gamma} \leq n,\quad k=r,\ldots, n,\quad j=0,\ldots,
(n/r)^{1/\gamma}\}, \qquad \gamma\geq 1 .
$$
Then the proposed version of the truncation method can be written as
\begin{equation} \label{ModVer}
\mathcal{D}_{n,\gamma}^{(r,0)} f^\delta(t,\tau) = \sum_{k=r}^{n} \sum_{j=0}^{(n/k)^{1/\gamma}} \langle f^\delta,
T_{k,j}\rangle T^{(r)}_k(t)T_j(\tau).
\end{equation}

The parameters $n$ and $\gamma$ in (\ref{ModVer}) should be chosen depending on $\delta$, $p$, $s$, $r$ and $\overline{\mu}$
so as to minimize the error of the method $\mathcal{D}_{n,\gamma}^{(r,0)}$, which can be written as
\begin{equation}\label{fullError}
f^{(r,0)}(t,\tau)-\mathcal{D}_{n,\gamma}^{(r,0)} f^\delta(t,\tau)=
\left(f^{(r,0)}(t,\tau)-\mathcal{D}_{n,\gamma}^{(r,0)} f(t,\tau)\right)+\left(\mathcal{D}_{n,\gamma}^{(r,0)}
f(t,\tau)-\mathcal{D}_{n,\gamma}^{(r,0)} f^\delta(t,\tau)\right).
\end{equation}
For the first difference on the right-hand side of (\ref{fullError}), the representation holds true
\begin{equation}\label{Bound_err}
 f^{(r,0)}(t,\tau)-\mathcal{D}_{n,\gamma}^{(r,0)} f(t,\tau)= \triangle_{1}(t,\tau)+\triangle_{2}(t,\tau) ,
\end{equation}
where
\begin{equation}\label{Triangle_1HC}
\triangle_{1}(t,\tau)= \sum_{k=n+1}^{\infty} \, \sum_{j=0}^{\infty} \langle f,
T_{k,j}\rangle T^{(r)}_k(t)T_j(\tau),
\end{equation}
\begin{equation}\label{Triangle_2HC}
 \triangle_{2}(t,\tau)= \sum_{k=r}^{n} \, \sum_{j>(n/k)^{1/\gamma}} \langle f,
 T_{k,j}\rangle T^{(r)}_k(t)T_j(\tau) .
\end{equation}

For our calculations, we need the differentiation formula for Chebyshev polynomials. It is known \cite{MasonHandscomb2002} that:
\begin{equation}\label{Chebyshev_diff_2D}
 \frac{d}{dt} T_{k}(t) = 2 k
\sum_{l=0}^{k-1} \zeta_l T_{l}(t) ,
\quad k\in\mathbb{N} ,\qquad \zeta_0=\sqrt{2}, \quad \zeta_l=1 \text{ for } l\in\mathbb{N} ,
\end{equation}
where the summation is extended over only those terms for which $k+l$ is odd.

Let us estimate the error of the method (\ref{ModVer}) in the metric of $L_{2,\omega}$. A bound for difference (\ref{Bound_err})
is contained in the following statement.
\vspace{2mm}

\begin{lemma}\label{lemma_BoundErrHC}
Let $f\in W^{\overline{\mu}}_{s,2}$, $1\leq s< \infty$, $\mu_1>2r-1/s+1/2$, $\mu_2>\mu_1-2r$.
 Then for $1\leq \gamma < \frac{\mu_2+1/s-1/2}{\mu_1-2r+1/s-1/2}$ it holds
$$
  \|f^{(r,0)}-\mathcal{D}_{n,\gamma}^{(r,0)} f\|_{L_{2,\omega}}\leq c\, \|f\|_{s,\overline{\mu}}\, n^{-\mu_1+2r-1/s+1/2} .
$$

\end{lemma}

\textit{Proof.} Using $r$ times the formula (\ref{Chebyshev_diff_2D}), from (\ref{Triangle_1HC}) we have
$$
\triangle_{1}(t,\tau)
= 2^r \sum_{k=n+1}^{\infty} \sum_{j=0}^{\infty} k \:\langle f, T_{k,j}\:\rangle
\, T_{j}(\tau)
\mathop{{\sum}}\limits_{l_1=r-1}^{k-1}  l_1  \mathop{{\sum}}\limits_{l_2=r-2}^{l_1-1}
l_2 \ldots \mathop{{\sum}}\limits_{l_{r-1}=1}^{l_{r-2}-1}  l_{r-1}
\mathop{{\sum}}\limits_{l_{r}=0}^{l_{r-1}-1}
\zeta_{l_{r}} \,T_{l_r}(t) .
$$

Note that in the representation $\triangle_{1}$ only those terms take place for which all indexes
$l_1+k, l_2+l_1,...,l_r+l_{r-1}$ are odd.

Further, we change the order of summation and get
$$
\triangle_{1}(t,\tau)= \triangle_{11}(t,\tau)+\triangle_{12}(t,\tau),
$$
where
\begin{equation}\label{triangle_{11}}
\triangle_{11}(t,\tau)= 2^r \sum_{j=0}^{\infty} T_{j}(\tau) \mathop{{\sum}}\limits_{l_r=0}^{n-r+1}
\zeta_{l_r} \: T_{l_r}(t)
 \sum_{k=n+1}^{\infty} k \:\langle f, T_{k,j}\:
 \rangle B^r_{k} ,
\end{equation}
\begin{equation}\label{triangle_{12}}
\triangle_{12}(t,\tau)= 2^r \sum_{j=0}^{\infty} T_{j}(\tau) \mathop{{\sum}}\limits_{l_r=n-r+2}^{\infty}
\zeta_{l_r}\: T_{l_r}(t) \, \sum_{k=l_r+r}^{\infty}\, k \:\langle f, T_{k,j}\:
 \rangle B^r_{k}
\end{equation}
and
\begin{equation}\label{axular1}
B^r_{k}:=\mathop{{\sum}}\limits_{l_1=l_r+r-1}^{k-1}  l_1  \mathop{{\sum}}\limits_{l_2=l_r+r-2}^{l_1-1}
l_2 \ldots \mathop{{\sum}}\limits_{l_{r-1}=l_r+1}^{l_{r-2}-1}  l_{r-1}  \leq c k^{2(r-1)} .
\end{equation}

In the future, all terms in the error representation (\ref{fullError}) of method $\mathcal{D}_{n,\gamma}^{(r,0)}$, for example,
(\ref{Triangle_1HC}) and (\ref{Triangle_2HC}), will be estimated using the same proof scheme, consisting of 5 steps a)-e).
First, let's bound $\triangle_{11}(t,\tau)$ in $L_{2,\omega}$-norm.\\
a) Taking into account (\ref{triangle_{11}}) and (\ref{axular1}), it is easy to write
\begin{equation}  \label{est_11}
\|\triangle_{11}\|_{L_{2,\omega}}^2 \leq c \sum_{j=0}^{\infty} \underline{j}^{-2\mu_2} \mathop{{\sum}}\limits_{l_r=0}^{n-r+1} \left( \sum_{k=n+1}^{\infty} k^{\mu_1} \underline{j}^{\mu_2} |\langle f, T_{k,j}\rangle|\ k^{-\mu_1+2r-1}\right)^2.
\end{equation}
b) Using Hölder's inequality for the sum over $k$, we have for $s > 1$
$$
\left(\sum_{k=n+1}^{\infty} k^{\mu_1} \underline{j}^{\mu_2} |\langle f, T_{k,j}\rangle|\ k^{-\mu_1+2r-1}\right)^2 \leq  \left( \sum_{k=n+1}^{\infty} k^{s\mu_1} \underline{j}^{s\mu_2} |\langle f, T_{kj}\rangle |^s\right)^{\frac{2}{s}}
\left( \sum_{k=n+1}^{\infty} k^{(-\mu_1+2r-1)s/(s-1)} \right)^{\frac{2(s-1)}{s}} .
$$
c) Further, under the condition $\mu_1>2r-1/s$ we obtain
\begin{eqnarray}   \label{c)}
\|\triangle_{11}\|_{L_{2,\omega}}^2  & \leq & c  \, n^{-2(\mu_1-2r+1)+2(s-1)/s} \sum_{j=0}^{\infty} \underline{j}^{-2\mu_2} \left( \sum_{k=n+1}^{\infty} k^{s\mu_1}
\underline{j}^{s\mu_2} |\langle f, T_{k,j}\rangle |^s\right)^{\frac{2}{s}}\mathop{{\sum}}\limits_{l_r=0}^{n-r+1}1
\nonumber \\ &
= & c n^{-2(\mu_1-2r+1/s-1/2)} \sum_{j=0}^{\infty} \underline{j}^{-2\mu_2} \left( \sum_{k=n+1}^{\infty} k^{s\mu_1}
\underline{j}^{s\mu_2} |\langle f, T_{k,j}\rangle |^s\right)^{\frac{2}{s}} .
\end{eqnarray}
For $s=1$, the steps b), c) of the scheme need to be modified. \\
b') Namely, from (\ref{est_11}) we get for $\mu_1>2r-1$
$$
\left(\sum_{k=n+1}^{\infty} k^{\mu_1} \underline{j}^{\mu_2} |\langle f, T_{k,j}\rangle|\ k^{-\mu_1+2r-1}\right)^2 \leq
c\, n^{-2(\mu_1-2r+1)}\, \left( \sum_{k=n+1}^{\infty} k^{\mu_1} \underline{j}^{\mu_2} |\langle f, T_{kj}\rangle |\right)^{2} .
$$
c') Then
$$
\|\triangle_{11}\|_{L_{2,\omega}}^2  \leq c  \, n^{-2(\mu_1-2r+1/2)} \sum_{j=0}^{\infty} \underline{j}^{-2\mu_2}
\left( \sum_{k=n+1}^{\infty} k^{\mu_1} \underline{j}^{\mu_2} |\langle f, T_{kj}\rangle |\right)^{2} .
$$
d) Now using Hölder's inequality for the sum over $j$, we have for $s > 2$
$$
\sum_{j=0}^{\infty} \underline{j}^{-2\mu_2} \left( \sum_{k=n+1}^{\infty} k^{s\mu_1}
\underline{j}^{s\mu_2} |\langle f, T_{k,j}\rangle |^s\right)^{\frac{2}{s}} \leq \left( \sum_{j=0}^{\infty}   \sum_{k=n+1}^{\infty} k^{s\mu_1}
\underline{j}^{s\mu_2} |\langle f, T_{k,j}\rangle |^s\right)^{\frac{2}{s}}\left( \sum_{j=0}^{\infty}\underline{j}^{-2\mu_2 \frac{s}{s-2}}\right)^{\frac{s-2}{s}} .
$$
e) Finally, under the condition $\mu_2\geq(s-2)/(2s)$ we obtain
\begin{equation}  \label{estDelta11}
\|\triangle_{11}\|_{L_{2,\omega}}^2\leq c \|f\|_{s,\overline{\mu}}^2 \, n^{-2(\mu_1-2r+1/s-1/2)} .
\end{equation}
For $s=2$, estimate (\ref{estDelta11}) follows directly from (\ref{c)}).

In case $1\leq s<2$, the last two steps of the scheme need to be modified. \\
d')--e') Namely, due to $\|f\|_{s,\overline{\mu}}\le 1$, we get
$$
\|\triangle_{11}\|_{L_{2,\omega}}^2\leq c \ n^{-2(\mu_1-2r+1/s-1/2)} \sum_{j=0}^{\infty} \underline{j}^{-2\mu_2}\times \left( \sum_{k=n+1}^{\infty} k^{s\mu_1} \underline{j}^{s\mu_2} |\langle f, T_{kj}\rangle |^s\right)^{\frac{2}{s}}
$$
$$
\leq c \, \|f\|_{s,\overline{\mu}}^{2-s} \,n^{-2(\mu_1-2r+1/s-1/2)}\sum_{j=0}^{\infty} \underline{j}^{-2\mu_2} 
\sum_{k=n+1}^{\infty} k^{s\mu_1} \underline{j}^{s\mu_2} |\langle f, T_{kj}\rangle |^s
$$
$$
\leq c\, \|f\|_{s,\overline{\mu}}^2 \, n^{-2(\mu_1-2r+1/s-1/2)} .
$$

Repeating all the steps of the proof scheme, we can estimate $\|\triangle_{12}\|_{L_{2,\omega}}$
for $\mu_1>2r-1/s+1/2$, $1\le s<\infty$ and $\mu_2>1/2-1/s$ as:
$$
\|\triangle_{12}\|_{L_{2,\omega}} \leq c\, \|f\|_{s,\overline{\mu}}\,  n^{-\mu_1+2r-1/s+1/2} .
$$

Further, using (\ref{Chebyshev_diff_2D}) $r$ times for differentiation, we get from (\ref{Triangle_2HC}):
$$
\triangle_{2}(t,\tau)= 2^r \sum_{k=r}^{n}\,
\sum_{j>(n/k)^{1/\gamma}} k\, \langle f,
 T_{k,j}\rangle T_j(\tau)
$$
$$
\times \mathop{{\sum}}\limits_{l_1=r-1}^{k-1}  l_1  \mathop{{\sum}}\limits_{l_2=r-2}^{l_1-1} l_2
\ldots \mathop{{\sum}}\limits_{l_{r-1}=1}^{l_{r-2}-1} l_{r-1}
 \mathop{{\sum}}\limits_{l_{r}=0}^{l_{r-1}-1} \zeta_{l_{r}}\: T_{l_r}(t) .
 $$
Now we change the order of summation for $l_r$ ($r-1$ times) and get
\begin{equation}\label{triangle_{2}}
\triangle_{2}(t,\tau) = 2^r\,
\sum_{k=r}^{n} \, \sum_{j>(n/k)^{1/\gamma}} k \langle f,
 T_{k,j}\rangle T_j(\tau)
\sum_{l_r=0}^{k-r} \zeta_{l_{r}}\: T_{l_r}(t) B^r_{k} ,
\end{equation}
where the quantity $B^r_{k}$ is defined by (\ref{axular1}).

As a result of changing the order of summation by $k$, $j$ and $l_r$, we have
$$
\triangle_{2}(t,\tau) = \triangle_{21}(t,\tau) + \triangle_{22}(t,\tau) + \triangle_{23}(t,\tau) ,
$$
where
\begin{equation}\label{triangle_{21}}
\triangle_{21}(t,\tau) = 2^r\, \sum_{j=1}^{(n/r)^{1/\gamma}}\, T_j(\tau)\, \sum_{l_r=0}^{n/j^\gamma-r}\, \zeta_{l_{r}}\, \: T_{l_r}(t) \, \sum_{k=n/j^\gamma}^{n}k\, \langle f,
 T_{k,j}\rangle \, B^r_{k} ,
\end{equation}
\begin{equation}\label{triangle_{22}}
\triangle_{22}(t,\tau) = 2^r\, \sum_{j=1}^{(n/r)^{1/\gamma}}\, T_j(\tau)\, \sum_{l_r=n/j^\gamma-r+1}^{n-r}\, \zeta_{l_{r}}\, \: T_{l_r}(t) \, \sum_{k=l_r+r}^{n}\, k\, \langle f,
 T_{k,j}\rangle\, B^r_{k} ,
\end{equation}
\begin{equation}\label{triangle_{23}}
\triangle_{23}(t,\tau) = 2^r \sum_{j=(n/r)^{1/\gamma}+1}^{\infty}\, T_j(\tau)\,
\sum_{l_r=0}^{n-r}\, \zeta_{l_{r}}\, \: T_{l_r}(t) \ \sum_{k=l_r+r}^{n}k\, \langle f,
 T_{k,j}\rangle\, B^r_{k} .
\end{equation}
Repeating all the steps of the proof scheme, we obtain
for $\mu_1>2r-1/s+1/2$, $\mu_2>1/2-1/s$, $1\le s<\infty$ and $\mu_2+1/s-1/2 > \gamma (\mu_1-2r+1/s-1/2)$:
$$
\|\triangle_{21}\|_{L_{2,\omega}} + \|\triangle_{22}\|_{L_{2,\omega}} + \|\triangle_{23}\|_{L_{2,\omega}}
\leq c\, \|f\|_{s,\overline{\mu}}\,  n^{-\mu_1+2r-1/s+1/2} .
$$

Lemma is completely proven.

  ${}$ \ \ \ \ \ \ \ \ \ \ \ \ \ \ \ \ \ \ \ \ \ \ \ \ \ \ \ \ \ \ \ \ \ \ \ \ \ \
  \ \ \ \ \ \ \ \ \ \ \ \ \ \ \ \ \ \ \ \ \ \
  \ \ \ \ \ \ \ \ \ \ \ \ \ \ \ \ \ \ \ \ \ \
  \ \ \ \ \ \ \ \ \ \ \ \ \ \ \ \ \ \ \ \ \ \
  \ \ \ \ \ \ \ \ \ \ \ \ \ \ \ \ \ \ \ \ \ \
  \ \ \ \ \ \ \ \ \ \ \ \ \ \ \ \ \ \ \ \ \ \
  \ \ $\Box$

\begin{lemma}\label{lemma_BoundPertL2}
Let the condition (\ref{perturbation}) be satisfied for $1\leq p \leq \infty$. Then for any function $f\in
L_{2,\omega}(Q)$ and $\gamma \geq 1$ it holds
 $$
 \|\mathcal{D}^{(r,0)}_{n,\gamma} f - \mathcal{D}^{(r,0)}_{n,\gamma} f^\delta\|_{L_{2,\omega}} \leq c\delta n^{2r-1/p+1/2} .
 $$
\end{lemma}

\textit{Proof.} Let's write down the following representation
the second difference on the right-hand side of (\ref{fullError}):
$$
\mathcal{D}^{(r,0)}_{n,\gamma} f(t,\tau) - \mathcal{D}^{(r,0)}_{n,\gamma} f^\delta(t,\tau)=
\sum_{k=r}^{n} \sum_{j=0}^{(n/k)^{1/\gamma}} \langle f-f^\delta, T_{k,j}\rangle
\frac{d^r T_k(t)}{d t^r}\,T_j(\tau).
$$
After differentiation $r$ times we get
$$
\mathcal{D}^{(r,0)}_{n,\gamma} f(t,\tau) - \mathcal{D}^{(r,0)}_{n,\gamma} f^\delta(t,\tau)= 2^r \sum_{k=r}^{n}\,
\sum_{j=0}^{(n/k)^{1/\gamma}} k\, \xi_{k,j} \  T_j(\tau)
$$
$$
\times \mathop{{\sum}}\limits_{l_1=r-1}^{k-1}  l_1 \: \mathop{{\sum}}\limits_{l_2=r-2}^{l_1-1} l_2 \:
\ldots \mathop{{\sum}}\limits_{l_{r-1}=1}^{l_{r-2}-1} l_{r-1} \:
 \mathop{{\sum}}\limits_{l_{r}=0}^{l_{r-1}-1} \zeta_{l_{r}}\: T_{l_r}(t) .
$$
Further, we change the order of summation and get
\begin{equation}   \label{2nd_term}
\mathcal{D}^{(r,0)}_{n,\gamma} f(t,\tau) - \mathcal{D}^{(r,0)}_{n,\gamma} f^\delta(t,\tau) = 2^r\,
\sum_{j=0}^{(n/r)^{1/\gamma}} \, T_j(\tau) \mathop{{\sum}}\limits_{l_r=0}^{n/\underline{j}^\gamma-r}
\zeta_{l_r}\, T_{l_r}(t)
\, \sum_{k=l_r+r}^{n/\underline{j}^\gamma} k\, \xi_{k,j}\:  B^r_{k} ,
\end{equation}
where $B^r_{k}$ is defined by (\ref{axular1}).

Now, using orthonormality of Chebyshev polynomials and (\ref{axular1}), we have
$$
\|\mathcal{D}^{(r,0)}_{n,\gamma} f - \mathcal{D}^{(r,0)}_{n,\gamma} f^\delta\|_{L_{2,\omega}}^2 =
\sum_{j=0}^{(n/r)^{1/\gamma}} \, \mathop{{\sum}}\limits_{l_r=0}^{n/\underline{j}^\gamma-r}
\Bigg(\sum_{k=l_r+r}^{n/\underline{j}^\gamma} k\, \xi_{k,j}\:  B^r_{k}\Bigg)^2
$$
$$
\leq c\,
\sum_{j=0}^{(n/r)^{1/\gamma}} \, \mathop{{\sum}}\limits_{l_r=0}^{n/\underline{j}^\gamma-r}
\Bigg(\sum_{k=l_r+r}^{n/\underline{j}^\gamma} |\xi_{k,j}| k^{2r-1}\Bigg)^2.
$$
Using Hölder inequality over $k$, we have for $1<p<\infty$
$$
\sum_{k=l_r+r}^{n/\underline{j}^\gamma} |\xi_{k,j}| k^{2r-1} \leq \Bigg(\sum_{k=l_r+r}^{n/\underline{j}^\gamma}|\xi_{k,j}|^p\Bigg)^{1/p} \Bigg(\sum_{k=l_r+r}^{n/\underline{j}^\gamma}k^{(2r-1)p/(p-1)}\Bigg)^{(p-1)/p}
$$
$$
 \leq c \delta n^{(2r-1/p)} j^{-\gamma(2r-1/p)} .
$$
Then
$$
\begin{array}{ll}
\|\mathcal{D}^{(r,0)}_{n,\gamma} f - \mathcal{D}^{(r,0)}_{n,\gamma} f^\delta\|_{L_{2,\omega}}^2
& \leq c \delta^2 n^{2(2r-1/p)} \sum\limits_{j=0}^{(n/r)^{1/\gamma}}j^{-2\gamma(2r-1/p)} \, \mathop{{\sum}}\limits_{l_r=0}^{n/\underline{j}^\gamma-r}1
\\\\
& \leq c \delta^2 n^{2(2r-1/p)+1} \sum\limits_{j=0}^{(n/r)^{1/\gamma}}j^{-\gamma(2(2r-1/p)+1)}.
\end{array}
$$
Thus,
$$
\|\mathcal{D}^{(r,0)}_{n,\gamma} f - \mathcal{D}^{(r,0)}_{n,\gamma} f^\delta\|_{L_{2,\omega}}^2 \leq c  \delta^2 n^{4r-2/p+1} ,
$$
which was required to prove.

In the case of $p=1$ and $p=\infty$, the assertion of Lemma is proved similarly. \vspace{0.1in}
\vskip -4mm
  ${}$ \ \ \ \ \ \ \ \ \ \ \ \ \ \ \ \ \ \ \ \ \ \ \ \ \ \ \ \ \ \ \ \ \ \ \ \ \ \
  \ \ \ \ \ \ \ \ \ \ \ \ \ \ \ \ \ \ \ \ \ \
  \ \ \ \ \ \ \ \ \ \ \ \ \ \ \ \ \ \ \ \ \ \
  \ \ \ \ \ \ \ \ \ \ \ \ \ \ \ \ \ \ \ \ \ \
  \ \ \ \ \ \ \ \ \ \ \ \ \ \ \ \ \ \ \ \ \ \
  \ \ \ \ \ \ \ \ \ \ \ \ \ \ \ \ \ \ \ \ \ \
  \ \ $\Box$

\begin{theorem} \label{Th1}
Let $f\in BW^{\bar\mu}_{s,2}$, $1\leq s< \infty$, $\mu_1>2r-1/s+1/2$, $\mu_2>\mu_1-2r$, and let the condition
(\ref{perturbation}) be satisfied for $1\leq p \leq \infty$.
Then for $n\asymp \delta^{-\frac{1}{\mu_1-1/p+1/s}}$ and $1\leq \gamma < \frac{\mu_2+1/s-1/2}{\mu_1-2r+1/s-1/2}$ it holds
$$
\|f^{(r,0)} - \mathcal{D}^{(r,0)}_{n,\gamma} f^\delta\|_{L_{2,\omega}} \leq c \delta^{\frac{\mu_1-2r+1/s-1/2}{\mu_1-1/p+1/s}} .
$$
\end{theorem}

{\bf Proof.}
Taking into account  Lemmas \ref{lemma_BoundErrHC} and \ref{lemma_BoundPertL2},  from (\ref{fullError})  we get
$$
\begin{array}{lcl}
\|f^{(r,0)}-\mathcal{D}_{n,\gamma}^{(r,0)} f^\delta\|_{L_{2,\omega}}
& \leq  &
\|f^{(r,0)}-\mathcal{D}_{n,\gamma}^{(r,0)} f\|_{L_{2,\omega}} +
\|\mathcal{D}^{(r,0)}_{n,\gamma} f - \mathcal{D}^{(r,0)}_{n,\gamma} f^\delta\|_{L_{2,\omega}}
\\\\
& \leq &
c\,\left(n^{-\mu_1+2r-1/s+1/2} + \delta\, n^{2r-1/p+1/2}\right)
\\\\
& = & c\, n^{2r+1/2} \left(n^{-\mu_1-1/s} + \delta\, n^{-1/p}\right) .
\end{array}
$$
Substituting the rule $n\asymp \delta^{-\frac{1}{\mu_1-1/p+1/s}}$ into the relation above completely proves Theorem.

\vskip 2mm

\begin{corollary} \label{Cor1}
\rm In the considered problem, the truncation method $\mathcal{D}^{(r,0)}_{n,\gamma}$ (\ref{ModVer})
achieves the accuracy
$O\left(\delta^{\frac{\mu_1-2r+1/s-1/2}{\mu_1-1/p+1/s}}\right)$
on the class $BW^{\bar\mu}_{s,2}$, $1\leq s< \infty$, $\mu_1>2r-1/s+1/2$, $\mu_2>\mu_1-2r$, and requires
$$
\mathrm{card}(\Gamma_{n,\gamma}) \asymp \left\{
\begin{array}{cl}
n \asymp \delta^{-\frac{1}{\mu_1-1/p+1/s}}, & \mbox{ if }\ 1 < \gamma < \frac{\mu_2+1/s-1/2}{\mu_1-2r+1/s-1/2}, \\
n\ \ln n \asymp \delta^{-\frac{1}{\mu_1-1/p+1/s}} \ln \frac{1}{\delta}, & \mbox{ if }\ \gamma=1,
\end{array}
\right.
$$
perturbed Fourier-Chebyshev coefficients.
\end{corollary}

\vskip 2mm

\section{Error estimate for truncation method in the metric of $C$}
Now we study the approximation properties of $\mathcal{D}^{(r,0)}_{n,\gamma}$ in the uniform metric.

\begin{lemma}\label{lemma_BoundErrHCC}
Let $f\in W^{\overline{\mu}}_{s,2}$, $1\leq s< \infty$, $\mu_1>2r-1/s+1$, $\mu_2>\mu_1-2r$.
 Then for $1\leq \gamma < \frac{\mu_2+1/s-1}{\mu_1-2r+1/s-1}$ it holds
$$
  \|f^{(r,0)}-\mathcal{D}^{(r,0)}_{n,\gamma} f\|_{C}\leq c\|f\|_{s,\overline{\mu}} \, n^{-\mu_1+2r-1/s+1} .
$$
\end{lemma}

\textit{Proof.}
As in the previous Section, all terms $\triangle_{11}$ (\ref{triangle_{11}}), $\triangle_{12}$ (\ref{triangle_{12}}),
$\triangle_{21}$ (\ref{triangle_{21}}), $\triangle_{22}$ (\ref{triangle_{22}}), $\triangle_{23}$ (\ref{triangle_{23}})
in the error representation of method $\mathcal{D}^{(r,0)}_{n,\gamma}$ (\ref{ModVer})
will be estimated using the same proof scheme, consisting of 5 steps a)-e).

First, let's bound $\triangle_{11}(t,\tau)$ in $C$-norm.\\
a) Taking into account (\ref{axular1}), it is easy to write
$$
\|\triangle_{11}\|_{C} \leq c \sum_{j=0}^{\infty} \underline{j}^{-\mu_2} \mathop{{\sum}}\limits_{l_r=0}^{n-r+1} \left( \sum_{k=n+1}^{\infty} k^{\mu_1} \underline{j}^{\mu_2} |\langle f, T_{k,j}\rangle|\ k^{-\mu_1+2r-1}\right).
$$

b) Using Hölder's inequality for the sum over $k$, we have for $s > 1$
$$
\sum_{k=n+1}^{\infty} k^{\mu_1} \underline{j}^{\mu_2} |\langle f, T_{k,j}\rangle|\ k^{-\mu_1+2r-1} \leq \left( \sum_{k=n+1}^{\infty} k^{s\mu_1} \underline{j}^{s\mu_2} |\langle f, T_{kj}\rangle |^s\right)^{\frac{1}{s}}
\left( \sum_{k=n+1}^{\infty} k^{(-\mu_1+2r-1)s/(s-1)} \right)^{\frac{s-1}{s}}.
$$

c) Further, under the condition $\mu_1>2r-1/s$ we obtain
$$
\|\triangle_{11}\|_{C}  \leq  c  \, n^{-\mu_1+2r-1+(s-1)/s} \sum_{j=0}^{\infty} \underline{j}^{-\mu_2} \left( \sum_{k=n+1}^{\infty} k^{s\mu_1}
\underline{j}^{s\mu_2} |\langle f, T_{k,j}\rangle |^s\right)^{\frac{1}{s}}\mathop{{\sum}}\limits_{l_r=0}^{n-r+1}1
$$
$$
= c n^{-\mu_1+2r-1/s+1} \sum_{j=0}^{\infty} \underline{j}^{-\mu_2} \left( \sum_{k=n+1}^{\infty} k^{s\mu_1}
\underline{j}^{s\mu_2} |\langle f, T_{k,j}\rangle |^s\right)^{\frac{1}{s}}.
$$

For $s=1$, the steps b), c) of the scheme need to be modified. \\
b') Namely, from a) we get for $\mu_1>2r-1$
$$
\sum_{k=n+1}^{\infty} k^{\mu_1} \underline{j}^{\mu_2} |\langle f, T_{k,j}\rangle|\ k^{-\mu_1+2r-1} \leq
c\, n^{-(\mu_1-2r+1)}\, \sum_{k=n+1}^{\infty} k^{\mu_1} \underline{j}^{\mu_2} |\langle f, T_{kj}\rangle |.
$$
c') Then
$$
\|\triangle_{11}\|_{C}  \leq c  \, n^{-\mu_1+2r} \sum_{j=0}^{\infty} \underline{j}^{-\mu_2}
\sum_{k=n+1}^{\infty} k^{\mu_1} \underline{j}^{\mu_2} |\langle f, T_{kj}\rangle |.
$$

d) Now using Hölder's inequality for the sum over $j$, we have for $s > 1$
$$
\sum_{j=0}^{\infty} \underline{j}^{-\mu_2} \left( \sum_{k=n+1}^{\infty} k^{s\mu_1}
\underline{j}^{s\mu_2} |\langle f, T_{k,j}\rangle |^s\right)^{\frac{1}{s}} \leq \left( \sum_{j=0}^{\infty}   \sum_{k=n+1}^{\infty} k^{s\mu_1}
\underline{j}^{s\mu_2} |\langle f, T_{k,j}\rangle |^s\right)^{\frac{1}{s}}\left( \sum_{j=0}^{\infty}\underline{j}^{-\mu_2 \frac{s}{s-1}}\right)^{\frac{s-1}{s}}.
$$

e) Finally, under the condition $\mu_2>(s-1)/s$ we obtain
$$
\|\triangle_{11}\|_{C}\leq c \|f\|_{s,\overline{\mu}} \, n^{-\mu_1+2r-1/s+1}.
$$

For $s=1$, estimate follows directly from c'):
$$
\|\triangle_{11}\|_{C}\leq c \ n^{-\mu_1+2r} \sum_{j=0}^{\infty} \underline{j}^{-\mu_2} \sum_{k=n+1}^{\infty} k^{\mu_1} \underline{j}^{\mu_2} |\langle f, T_{kj}\rangle |
\leq c \|f\|_{1,\overline{\mu}} \, n^{-\mu_1+2r}.
$$

Again using the same proof scheme for the remaining terms from the error representation, we obtain
for $\mu_1>2r-1/s+1$, $(\mu_2+1/s-1)/\gamma>\mu_1-2r+1/s$ and $\mu_2>1-1/s$:
$$
\|\triangle_{12}\|_{C} \leq c\, \|f\|_{s,\overline{\mu}} \, n^{-\mu_1+2r-1/s+1} ,
\qquad
\|\triangle_{21}\|_{C} \leq c \|f\|_{s,\overline{\mu}} \, n^{-\mu_1+2r-1/s+1} ,
$$
$$
\|\triangle_{22}\|_{C} \leq c \|f\|_{s,\overline{\mu}} \, n^{-\mu_1+2r-1/s+1} ,
\qquad
\|\triangle_{23}\|_{C} \leq c \|f\|_{s,\overline{\mu}} \, n^{-\mu_1+2r-1/s+1} .
$$

Combining the estimates for all terms, the Lemma is completely proven.
\vphantom{AAA}
\\
\ \ \ \ \ \ \ \ \ \ \ \ \ \ \ \ \ \ \ \ \ \ \ \ \ \ \ \ \ \ \ \ \ \ \ \ \ \
\ \ \ \ \ \ \ \ \ \ \ \ \ \ \ \ \ \ \ \ \ \
\ \ \ \ \ \ \ \ \ \ \ \ \ \ \ \ \ \ \ \ \ \
\ \ \ \ \ \ \ \ \ \ \ \ \ \ \ \ \ \ \ \ \ \
\ \ \ \ \ \ \ \ \ \ \ \ \ \ \ \ \ \ \ \ \ \
\ \ \ \ \ \ \ \ \ \ \ \ \ \ \ \ \ \ \ \ \ \
\ \ $\Box$

\begin{lemma}\label{lemma_BoundPertHCC}
Let the condition (\ref{perturbation}) be satisfied for $1\leq p \leq \infty$. Then for any function $f\in
L_{2,\omega}(Q)$ and $\gamma \geq 1$ it holds
 $$
 \|\mathcal{D}^{(r,0)}_{n,\gamma} f - \mathcal{D}^{(r,0)}_{n,\gamma} f^\delta\|_{C} \leq c \delta n^{2r-1/p+1} .
 $$
\end{lemma}

\textit{Proof.}
Let us recall (see (\ref{2nd_term})) that for the second term on the right-hand side of (\ref{fullError})
the representation is valid
$$
\mathcal{D}^{(r,0)}_{n,\gamma} f(t,\tau) - \mathcal{D}^{(r,0)}_{n,\gamma} f^\delta(t,\tau) = 2^r\,
\sum_{j=0}^{(n/r)^{1/\gamma}} \, T_j(\tau) \mathop{{\sum}}\limits_{l_r=0}^{n/\underline{j}^\gamma-r}
\zeta_{l_r}\, T_{l_r}(t)
\, \sum_{k=l_r+r}^{n/\underline{j}^\gamma} k\, \xi_{k,j}\:  B^r_{k} .
$$

Using (\ref{axular1}), we have
$$
\|\mathcal{D}^{(r,0)}_{n,\gamma} f - \mathcal{D}^{(r,0)}_{n,\gamma} f^\delta\|_{C} \leq c\,
\sum_{j=0}^{(n/r)^{1/\gamma}} \, \mathop{{\sum}}\limits_{l_r=0}^{n/\underline{j}^\gamma-r}
\sum_{k=l_r+r}^{n/\underline{j}^\gamma} |\xi_{k,j}| k^{2r-1} .
$$
Let $1<p<\infty$ first. By virtue of Hölder's inequality, we obtain
$$
\sum_{k=l_r+r}^{n/\underline{j}^\gamma} |\xi_{k,j}| k^{2r-1} \leq \Bigg(\sum_{k=l_r+r}^{n/\underline{j}^\gamma}|\xi_{k,j}|^p\Bigg)^{1/p} \Bigg(\sum_{k=l_r+r}^{n/\underline{j}^\gamma}k^{(2r-1)p/(p-1)}\Bigg)^{(p-1)/p}
$$
$$
 \leq c \delta n^{(2r-1/p)} j^{-\gamma(2r-1/p)} .
$$
Then
$$
\|\mathcal{D}^{(r,0)}_{n,\gamma} f - \mathcal{D}^{(r,0)}_{n,\gamma} f^\delta\|_{C} \leq c \delta n^{(2r-1/p)} \sum_{j=0}^{(n/r)^{1/\gamma}}j^{-\gamma(2r-1/p)} \, \mathop{{\sum}}\limits_{l_r=0}^{n/\underline{j}^\gamma-r}\,1
$$
$$
\leq c \delta n^{2r-1/p+1} \sum_{j=0}^{(n/r)^{1/\gamma}}j^{-\gamma(2r-1/p+1)} .
$$
Thus,
$$
\|\mathcal{D}^{(r,0)}_{n,\gamma} f - \mathcal{D}^{(r,0)}_{n,\gamma} f^\delta\|_{C} \leq c  \delta n^{2r-1/p+1} ,
$$
which was required to prove.

In the case of $p=1$ and $p=\infty$, the assertion of Lemma is proved similarly. \vspace{0.1in}

\vskip -4mm

  ${}$ \ \ \ \ \ \ \ \ \ \ \ \ \ \ \ \ \ \ \ \ \ \ \ \ \ \ \ \ \ \ \ \ \ \ \ \ \ \
  \ \ \ \ \ \ \ \ \ \ \ \ \ \ \ \ \ \ \ \ \ \
  \ \ \ \ \ \ \ \ \ \ \ \ \ \ \ \ \ \ \ \ \ \
  \ \ \ \ \ \ \ \ \ \ \ \ \ \ \ \ \ \ \ \ \ \
  \ \ \ \ \ \ \ \ \ \ \ \ \ \ \ \ \ \ \ \ \ \
  \ \ \ \ \ \ \ \ \ \ \ \ \ \ \ \ \ \ \ \ \ \
  \ \ $\Box$

\begin{theorem} \label{Th2}
Let $f\in BW^{\bar\mu}_{s,2}$, $1\leq s< \infty$, $\mu_1>2r-1/s+1$, $\mu_2>\mu_1-2r$, and let the condition
(\ref{perturbation}) be satisfied for $1\leq p \leq \infty$.
Then for $n\asymp \delta^{-\frac{1}{\mu_1-1/p+1/s}}$ and $1\leq \gamma < \frac{\mu_2+1/s-1}{\mu_1-2r+1/s-1}$ it holds
$$
\|f^{(r,0)} - \mathcal{D}^{(r,0)}_{n,\gamma} f^\delta\|_{C} \leq c \delta^{\frac{\mu_1-2r+1/s-1}{\mu_1-1/p+1/s}} .
$$
\end{theorem}

{\bf Proof.}
The combination of Lemmas (\ref{lemma_BoundErrHCC}) and (\ref{lemma_BoundPertHCC}) yields the statement of Theorem.

\vskip 2mm

\begin{corollary} \label{Cor2}
\rm In the considered problem, the truncation method $\mathcal{D}^{(r,0)}_{n,\gamma}$ (\ref{ModVer})
achieves the accuracy
$O\left(\delta^{\frac{\mu_1-2r+1/s-1}{\mu_1-1/p+1/s}}\right)$
on the class $BW^{\bar\mu}_{s,2}$, $1\leq s< \infty$, $\mu_1>2r-1/s+1$, $\mu_2>\mu_1-2r$, and requires
$$
\mathrm{card}(\Gamma_{n,\gamma}) \asymp \left\{
\begin{array}{cl}
n \asymp \delta^{-\frac{1}{\mu_1-1/p+1/s}}, & \mbox{ if }\ 1 < \gamma < \frac{\mu_2+1/s-1}{\mu_1-2r+1/s-1}, \\
n\ \ln n \asymp \delta^{-\frac{1}{\mu_1-1/p+1/s}} \ln \frac{1}{\delta}, & \mbox{ if }\ \gamma=1,
\end{array}
\right.
$$
perturbed Fourier-Chebyshev coefficients.
\end{corollary}

\vskip 2mm

\section{Truncation method. Error estimation in $L_{q,\omega}$ metric}
\vspace{2mm}
By $\Pi_{n,m}$ we mean the set of all algebraic bivariate polynomials of order $n$
in the first variable and order $m$ in the second variable.
Recall that $\omega(t,\tau)=(1-t^2)^{-1/2} (1-\tau^2)^{-1/2}$.

\begin{lemma}  \label{T_6.6}
Let ${\bf P}\in \Pi_{n,m}$. Then for any $0<p'\le q\le \infty$ it holds
\begin{equation}   \label{Mark}
\|{\bf P}\|_{L_{q,\omega}} \le c\, (n m)^{\frac{1}{p'}-\frac{1}{q}}\, \|{\bf P}\|_{L_{p',\omega}} .
\end{equation}
\end{lemma}
{\bf Proof.} The proof of (\ref{Mark}) in the case $p'=q$ is trivial.
To establish the lemma in other cases we use the proof scheme
of theorems 6.1, 6.2 \cite{DT}.
So, first, we prove (\ref{T_6.6}) in the case $p'=2$, $q=\infty$.
To do this, we write down the expansion of an arbitrary ${\bf P}\in \Pi_{n,m}$
$$
{\bf P}(t,\tau) =
\sum_{k=0}^n\, \sum_{l=0}^m \langle {\bf P}, T_{k,l} \rangle\, T_{k,l}(t,\tau)  ,
$$
where
$$
T_{k,l}(t,\tau) = T_{k}(t)\, T_{l}(\tau)
$$
are tensor products of univariate Chebyshev polynomials, and
$$
\langle {\bf P}, T_{k,l} \rangle =
\int_Q \, {\bf P}(x,y)\, \omega(x,y)\, T_{k,l}(x,y)\, dx dy
$$
are Fourier-Chebyshev coefficients of the polynomial ${\bf P}$.

Next, using the H\"{o}lder inequality, we find
$$
\begin{array}{lcl}
\|{\bf P}\|_{L_\infty} & = & \sup\limits_{-1\le t,\tau\le 1}\,
\left|\sum\limits_{k=0}^n\, \sum\limits_{l=0}^m \langle {\bf P}, T_{k,l} \rangle\, T_{k,l}(t,\tau)\right|
\\\
& \le & \left(\sum\limits_{k=0}^n\, \sum\limits_{l=0}^m \, |\langle {\bf P}, T_{k,l} \rangle|^2\right)^{1/2}\,
\sup\limits_{-1\le t,\tau\le 1}\,
\left(\sum\limits_{k=0}^n\,  |T_{k}(t)|^2\, \sum\limits_{l=0}^m\, |T_{l}(\tau)|^2\right)^{1/2}
\\\
& \le & \|{\bf P}\|_{L_{2,\omega}}\,
\left(\sup\limits_{-1\le t\le 1}\, \sum\limits_{k=0}^n\,  |T_{k}(t)|^2\,
\sup\limits_{-1\le \tau\le 1}\, \sum\limits_{l=0}^m\, |T_{l}(\tau)|^2\right)^{1/2}
\\\
& \le & \frac{2}{\pi}\, \|{\bf P}\|_{L_{2,\omega}}\, \sqrt{(n+1)(m+1)} .
\end{array}
$$
Thus, we have for any ${\bf P}\in \Pi_{n,m}$
\begin{equation}   \label{Mark2}
\|{\bf P}\|_{L_{\infty}} \le c\, (n m)^{1/2}\, \|{\bf P}\|_{L_{2,\omega}} .
\end{equation}

Next, take an arbitrary $p'>0$ and an integer $j$ such that $j\ge p'/2$.
It is easy to see that for any ${\bf P}\in \Pi_{n,m}$ the following holds:
\begin{eqnarray}   \label{P_j}
\|{\bf P}^j\|_{L_{2,\omega}} & := &
\left(\int_Q \, |{\bf P}(t,\tau)|^{2j}\, \omega(t,\tau)\, dt d\tau\right)^{1/2}
= \left(\int_Q \, |{\bf P}(t,\tau)|^{2j-p'}\, |{\bf P}(t,\tau)|^{p'}\, \omega(t,\tau)\, dt d\tau\right)^{1/2}
\nonumber \\
& \le & \sup\limits_{-1\le t,\tau\le 1}\, |{\bf P}(t,\tau)|^{j-p'/2}\,
\left(\int_Q \, |{\bf P}(t,\tau)|^{p'}\, \omega(t,\tau)\, dt d\tau\right)^{1/2}
\nonumber \\\
& \le & \|{\bf P}\|_{L_{\infty}}^{j-p'/2} \,\, \|{\bf P}\|_{L_{p',\omega}}^{p'/2} .
\end{eqnarray}
The combination of (\ref{Mark2}), applied to the polynomial ${\bf P}^j$, and the inequality (\ref{P_j}) gives
$$
\begin{array}{ll}
\|{\bf P}^j\|_{L_\infty} & \le c\, (j^2 n m)^{1/2}
\|{\bf P}^j\|_{L_{2,\omega}} \\\\
& \le c\, (n m)^{1/2} \|{\bf P}\|_{L_{\infty}}^{j-p'/2} \,\, \|{\bf P}\|_{L_{p',\omega}}^{p'/2} .
\end{array}
$$
From here it follows
$$
\begin{array}{ll}
\|{\bf P}\|_{L_\infty}^{p'/2} & \le c\, (n m)^{1/2}\,
\frac{\|{\bf P}\|_{L_{\infty}}^{j}}{\|{\bf P}^j\|_{L_\infty}} \,
\, \|{\bf P}\|_{L_{p',\omega}}^{p'/2} \\\\
& \le c\, (n m)^{1/2}\,  \|{\bf P}\|_{L_{p',\omega}}^{p'/2}
\end{array}
$$
that is
\begin{equation} \label{L_infty}
\|{\bf P}\|_{L_\infty}
\le c\, (n m)^{1/p'}\,  \|{\bf P}\|_{L_{p',\omega}} .
\end{equation}
Thus, the lemma is proved for $0<p'< \infty$, $q=\infty$. \\
Finally, for any $0<p'\le q< \infty$ we get
\begin{equation} \label{P_Lq}
\begin{array}{ll}
\|{\bf P}\|_{L_{q,\omega}} & :=
\left(\int\limits_Q \, |{\bf P}(t,\tau)|^{q}\, \omega(t,\tau)\, dt d\tau\right)^{1/q}
\\\\
& = \left(\int\limits_Q \, |{\bf P}(t,\tau)|^{q-p'}\, |{\bf P}(t,\tau)|^{p'}\, \omega(t,\tau)\, dt d\tau\right)^{1/q}
\\\\
& \le \|{\bf P}\|_{L_{\infty}}^{1-p'/q} \,\, \|{\bf P}\|_{L_{p',\omega}}^{p'/q} .
\end{array}
\end{equation}
From (\ref{P_Lq}) using (\ref{L_infty}) we can obtain
$$
\begin{array}{ll}
\|{\bf P}\|_{L_{q,\omega}} &
\le c\, (n m)^{\frac{1}{p'}(1-\frac{p'}{q})}\,  \|{\bf P}\|_{L_{p',\omega}}^{1-p'/q}
\,  \|{\bf P}\|_{L_{p',\omega}}^{p'/q}
\\\\
& = c\, (n m)^{\frac{1}{p'}-\frac{1}{q}}\,  \|{\bf P}\|_{L_{p',\omega}} ,
\end{array}
$$
which proves the statement.
$\Box$

\begin{lemma}  \label{T_11.1}
Let $g\in L_{q,\omega}$, $2\le q< \infty$. Then it holds
\begin{equation}   \label{LitP}
\|g\|_{L_{q,\omega}} \le c\,
\left(\sum_{k=0}^\infty \, \sum_{l=0}^\infty \,
(k l)^{1-\frac{2}{q}}\, |\langle g, T_{k,l} \rangle|^2\right)^{1/2}\, .
\end{equation}
\end{lemma}
{\bf Proof.}

First, we write the decomposition of $g$ as follows
$$
g(t,\tau) =
\sum_{k=0}^\infty \, \sum_{l=0}^\infty \, g_{k,l}(t,\tau) ,
$$
where
$$
g_{k,l}(t,\tau) = \langle g, T_{k,l} \rangle\, T_{k,l}(t,\tau) .
$$
Note that for any $q\ge 2$ the following holds:
$$
\begin{array}{ll}
\left(\int\limits_Q \Big|\sum\limits_{k=0}^\infty \, \sum\limits_{l=0}^\infty \,
g_{k,l}(t,\tau)\Big|^q\, \omega(t,\tau) dt d\tau\right)^{2/q}
& \le
\left(\sum\limits_{k=0}^\infty \, \sum\limits_{l=0}^\infty \,
\int\limits_Q |g_{k,l}(t,\tau)|^q\, \omega(t,\tau) dt d\tau\right)^{2/q}
\\\\
& \le
\sum\limits_{k=0}^\infty \, \sum\limits_{l=0}^\infty \,
\left(\int\limits_Q |g_{k,l}(t,\tau)|^q\, \omega(t,\tau) dt d\tau\right)^{2/q}
\\\\
& =
\sum\limits_{k=0}^\infty \, \sum\limits_{l=0}^\infty \,
\|g_{k,l}\|_{L_{q,\omega}}^2 .
\end{array}
$$
In other words, we have
\begin{eqnarray}   \label{g_Lq}
\|g\|_{L_{q,\omega}} & := &
\left(\int\limits_Q |g(t,\tau)|^q\, \omega(t,\tau) dt d\tau\right)^{1/q}
\nonumber \\\  & \le &
\left(\sum\limits_{k=0}^\infty \, \sum\limits_{l=0}^\infty \,
\|g_{k,l}\|^2_{L_{q,\omega}}\right)^{1/2} .
\end{eqnarray}
Applying inequality (\ref{Mark2}) to (\ref{g_Lq}), we obtain
$$
\begin{array}{lcl}
\|g\|_{L_{q,\omega}} & \le & c\,
\left(
\sum\limits_{k=0}^\infty \,\, \sum\limits_{l=0}^\infty \,
(k l)^{2(\frac{1}{2}-\frac{1}{q})}\, \|g_{k,l}\|^2_{L_{2,\omega}}
\right)^{1/2}
\\\\
& = & c\,
\left(
\sum\limits_{k=0}^\infty \,\, \sum\limits_{l=0}^\infty \,
(k l)^{1-\frac{2}{q}}\, |\langle g, T_{k,l} \rangle|^2\ \right)^{1/2} .
\end{array}
$$
As was to be proved. $\Box$

\vskip 2mm

\begin{remark} \label{Jacobi}
The results of Lemmas \ref{T_6.6}, \ref{T_11.1} can be easily generalized to the case of Jacobi polynomials and
functions of any number of variables.
\end{remark}	
\vskip 2mm
Now we are able to estimate the accuracy of the method $\mathcal{D}^{(r,0)}_{n,\gamma}$ in the metric $L_{q,\omega}$.

\begin{lemma}  \label{lemma_FirstDif_Lq}
Let $f\in W^{\bar\mu}_{s,2}$, $1\leq s< \infty$, $2\le q\le\infty$, $\mu_1>2r-1/s-1/q+1$, $\mu_2>1-1/s-1/q$.
Then for any $ 1 \le \gamma < \frac{\mu_2+1/s+1/q-1}{\mu_1-2r+1/s+1/q-1}$
it holds
$$
\|f^{(r,0)}-\mathcal{D}_{n,\gamma}^{(r,0)} f\|_{L_{q,\omega}}\leq c\, \|f\|_{s,\overline{\mu}}\, n^{-\mu_1+2r-1/s-1/q+1} .
$$
\end{lemma}
{\bf Proof.}
Recall that
$$
f^{(r,0)}(t,\tau) - \mathcal{D}_{n,\gamma}^{(r,0)} f(t,\tau) =
\triangle_{11}(t,\tau) + \triangle_{12}(t,\tau) + \triangle_{21}(t,\tau) + \triangle_{22}(t,\tau) + \triangle_{23}(t,\tau),
$$
where the expansion terms $\triangle_{11}$ -- $\triangle_{23}$ are defined by the relations
(\ref{triangle_{11}}), (\ref{triangle_{12}}), (\ref{triangle_{21}})--(\ref{triangle_{23}}), respectively.

When estimating all these terms, we will follow the proof scheme from Lemma \ref{lemma_BoundErrHC}.
The only difference is that in step a) we now use inequality (\ref{LitP}), which allows us to move
from the norm in $L_{q,\omega}$ to the norm in $L_{2,\omega}$.

In this way we obtain the statement of Lemma. $\Box$

\begin{lemma}\label{lemma_SecDif_Lq}
Let the condition (\ref{perturbation}) be satisfied for $1\leq p \leq \infty$.  Then for any function $f\in
L_{2,\omega}$ and $2\le q< \infty$ it holds
$$
\|\mathcal{D}^{(r,0)}_{n,\gamma} f - \mathcal{D}^{(r,0)}_{n,\gamma} f^\delta\|_{L_{q,\omega}} \leq
c\, \delta\, n^{2r-1/p-1/q+1} .
$$
\end{lemma}
{\bf Proof.}
As before, for the second difference on the right-hand side of (\ref{fullError}), we will use the following representation
$$
\mathcal{D}_{n,\gamma}^{(r,0)} f(t,\tau) - \mathcal{D}_{n,\gamma}^{(r,0)} f^\delta(t,\tau) =
2^r\, \sum_{j=0}^{(n/r)^{1/\gamma}} \, T_j(\tau) \,
\sum_{l_r=0}^{n/\underline{j}^{\gamma}-r}\, \zeta_{l_r} T_{l_r}(t) \,
\sum_{k=l_r+r}^{n/\underline{j}^{\gamma}} \, k \, \xi_{k,j} \, B_k^r ,
$$
where $\xi_{k,j} = \langle f - f^\delta, T_{k,j} \rangle$.

Using (\ref{LitP}) and Hölder's inequality, we obtain for $1<p<\infty$:
$$
\|\mathcal{D}^{(r,0)}_{n,\gamma} f - \mathcal{D}^{(r,0)}_{n,\gamma} f^\delta\|_{L_{q,\omega}}
\le c\,
\left(\sum\limits_{j=0}^{(n/r)^{1/\gamma}}
\sum\limits_{l_r=0}^{n/\underline{j}^{\gamma}-r} (\underline{j l_r})^{1-2/q}
\big(\frac{n}{\underline{j}^\gamma}\big)^{2(2r-1/p)} \delta^2\right)^{1/2}
$$
$$
\le c\, \delta\, n^{2r-1/p-1/q+1}.
$$

For $p=1$ and $p=\infty$, the assertion follows similarly.
Thus, Lemma is proved. $\Box$

\vskip 2mm
The combination of Lemmas \ref{lemma_FirstDif_Lq} and \ref{lemma_SecDif_Lq} gives
\begin{theorem} \label{Th1_L_q}
Let $f\in BW^{\bar\mu}_{s,2}$, $1\leq s< \infty$, $2\le q< \infty$,
$\mu_1>2r-1/s-1/q+1$, $\mu_2>1-1/s-1/q$.
Then for any $ 1 \le \gamma < \frac{\mu_2+1/s+1/q-1}{\mu_1-2r+1/s+1/q-1}$
and $n \asymp \delta^{-\frac{1}{\mu_1-1/p+1/s}}$
it holds
$$
\|f^{(r,0)}-\mathcal{D}_{n,\gamma}^{(r,0)} f^\delta\|_{L_{q,\omega}}
\leq c\, \delta^{\frac{\mu_1-2r+1/s+1/q-1}{\mu_1-1/p+1/s}} .
$$
\end{theorem}

\vskip 2mm

\begin{corollary} \label{Cor3}
\rm In the considered problem, the truncation method $\mathcal{D}^{(r,0)}_{n,\gamma}$ (\ref{ModVer})
achieves the accuracy
$O\left(\delta^{\frac{\mu_1-2r+1/s+1/q-1}{\mu_1-1/p+1/s}}\right)$
on the class $BW^{\bar\mu}_{s,2}$, $1\leq s< \infty$, $2\le q< \infty$,
$\mu_1>2r-1/s-1/q+1$, $\mu_2>1-1/s-1/q$, and requires
$$
\mathrm{card}(\Gamma_{n,\gamma}) \asymp \left\{
\begin{array}{cl}
n \asymp \delta^{-\frac{1}{\mu_1-1/p+1/s}}, & \mbox{ if }\ 1 < \gamma < \frac{\mu_2+1/s+1/q-1}{\mu_1-2r+1/s+1/q-1}, \\
n\ \ln n \asymp \delta^{-\frac{1}{\mu_1-1/p+1/s}} \ln \frac{1}{\delta}, & \mbox{ if }\ \gamma=1,
\end{array}
\right.
$$
perturbed Fourier-Chebyshev coefficients.
\end{corollary}

\vskip 2mm
	
\begin{remark} \label{Legendre}
The problem of recovering partial derivatives of bivariate functions using the truncation method
was previously studied in \cite{Sol_Stas_UMZ2022}, and \cite{Sem_Sol_SIM2025}.
Both of these articles considered truncated Legendre's method.
A comparison of our results with those of predecessors shows that,
in the $C$-metric, Chebyshev polynomials provide a higher accuracy of numerical differentiation compared to Legendre polynomials.
On the other hand, in the Hilbert space metric, both versions of the truncation method
(with Chebyshev and Legendre polynomials) provide the same order of accuracy.
\end{remark}	

\vskip 2mm

\begin{remark} \label{Compar_opt}
Chebyshev polynomials in the problem of numerical differentiation were also used in \cite{Sem_Sol_arxiv},
where the case of functions of one variable was considered.
Note that the accuracy estimates of the truncation method found in our work and in \cite{Sem_Sol_arxiv}
in the metrics of Hilbert space and in $C$, as well as the estimates of the number of 
Fourier-Chebyshev coefficients involved, coincide in order.
Since in \cite{Sem_Sol_arxiv} it is proved that the mentioned estimates are unimprovable on Wiener classes of functions of one variable,
then through the embedding of classes of functions of different numbers of variables one can conclude
that the corresponding estimates found in Sections 2, 3 of our work are optimal in order also in the case of functions
from the class $BW^{\bar\mu}_{s,2}$.
\end{remark}

\end{document}